\documentclass[a4paper,11pt]{article}

\usepackage{a4wide}
\usepackage{amsmath}
\usepackage{amssymb, amsfonts}
\usepackage{amsthm}
\usepackage{color}
\usepackage{enumerate}
\usepackage{verbatim}
\usepackage{bbm}

\usepackage{tikz}
\usetikzlibrary{decorations.pathreplacing}

\newcommand  {\bb}[1]{\mathbb{#1}}

\renewcommand{\geq}{\geqslant}

\renewcommand{\leq}{\leqslant}

\newcommand  {\e}{\varepsilon}

\renewcommand{\phi}{\varphi}

\newcommand  {\tld}[1]{\widetilde{#1}}

\newcommand  {\E}{\bb{E}}

\renewcommand{\P}{\bb{P}}

\newcommand  {\hs}{\hspace{2mm}}
\newcommand  {\hsl}{\hspace{1mm}}
\newcommand  {\1}{\mathbbm{1}}

\newtheorem{thm}{Theorem}
\newtheorem{lem}[thm]{Lemma}

\newtheorem{prop}[thm]{Proposition}

\begin{document}

\title{Increasing paths in regular trees}

\author{
 Matthew I. Roberts
  \footnote{Department of Mathematical Sciences, University of Bath, Bath, BA2 7AY, UK. \newline Email: \texttt {mattiroberts@gmail.com}.}
 \and 
 Lee Zhuo Zhao
  \footnote{Statistical Laboratory, University of Cambridge, Cambridge, CB3 0WB, UK. \newline Email: \texttt {lzz20@statslab.cam.ac.uk}.}
}

\date{\today}

\maketitle

\begin{abstract}
We consider a regular $n$-ary tree of height $h$, for which every vertex except the root is labelled with an independent and identically distributed continuous random variable. Taking motivation from a question in evolutionary biology, we consider the number of paths from the root to a leaf along vertices with increasing labels. We show that if $\alpha = n/h$ is fixed and $\alpha > 1/e$, the probability that there exists such a path converges to $1$ as $h \to \infty$. This complements a previously known result that the probability converges to $0$ if $\alpha \leq 1/e$.
\end{abstract}


\section{Introduction}
Consider a regular $n$-ary tree of height $h$, where $n = \lfloor\alpha h\rfloor$. To each vertex except the root attach an independent and identically distributed continuous random variable. We ask whether there is a path from the root to a leaf whose labels only increase. Nowak and Krug \cite{NK} called this \emph{accessibility percolation} and showed that $\P(\text{there exists an increasing path}) \to 0$ as $n \to \infty$ if $\alpha \leq 1/e$, whereas if $\alpha > 1$ then there exists some $p > 0$ depending on $\alpha$ such that $\P(\text{there exists an increasing path}) > p$. We give a complete characterisation in terms of $\alpha$, showing that there is a phase transition at $\alpha = 1/e$.

\begin{thm}\label{T:mainthm}
Suppose that $n = \lfloor \alpha h\rfloor$. As $h \to \infty$,
\[
 \P(\text{there exists an increasing path}) \to
 \begin{cases}
 0 & \text{ if } \alpha\leq 1/e, \\
 1 &\text{ if } \alpha> 1/e.
 \end{cases}
\]
\end{thm}

Given the result of \cite{NK} mentioned above, it suffices to prove the second statement. In fact we will show that for any $\alpha > 1/e$, there exist $\delta > 0$ and $\eta > 0$ such that
\[
 \P(\text{there exist at least } \exp(\delta h) \text{ increasing paths}) \geq 1 - \exp(-\eta h).
\]
This suggests that we might be able to discover more around the critical point $1/e$, and indeed by essentially the same methods we are able to obtain the following finer result.

\begin{thm}\label{T:ext}
Suppose that $n = \left( \frac{1 + \beta_h}{e} \right) h$, where $\beta_h \to 0$ as $h\to\infty$. Then as $h \to \infty$,
\[
 \P(\hbox{there exists an increasing path}) \to
 \begin{cases} 0 &\hbox{ if } \log h - 2h \beta_h \to \infty, \\
 1 &\hbox{ if } h \beta_h / \log h \to \infty.
 \end{cases}
\]
\end{thm}

\subsection{Biological motivation}
Consider the following simplified model of evolution in a population. Each genetic type, or \emph{genotype}, in the population has an associated fitness. A particular genotype may give rise to multiple new genotypes through mutations, which either replace the original wild genotype or disappear from the population. For a haploid asexual population, the dynamics of evolution are governed by the population size $N$, the selection coefficient $s$ and the mutation rate $\mu$ \cite{FKVK}. We make the following two assumptions on these three parameters:

\begin{enumerate}
 \item $Ns \gg 1$. By a classical formula of Kimura \cite{Kim}, only mutations which give rise to a fitter genotype can replace the wild genotype and survive.
 \item $\mu$ is sufficiently small such that mutations arise and either replace the wild genotype or become extinct one at a time. Therefore, there can be at most two genotypes in the population at any given time, of which one is a direct mutant of the other.
\end{enumerate}

Together, these two assumptions form what is known in the evolutionary biology literature as the \emph{strong selection weak mutation} (SSWM) regime \cite{G, O}. Under such a setting, the only possible evolutionary paths of genotypes are ones with increasing fitness. In the evolutionary biology literature, these increasing paths are known as \emph{selectively accessible} \cite{FKVK, WDDH, WWC}.

To analyse the number of such paths, we also require the relationship between genotype and fitness. For this, we use the \emph{House of Cards} model \cite{Kin, KL}, in which every genotype has an independent and identically continuously distributed fitness. Since we only care about whether the fitnesses along a path are in increasing order, as long as the random variables are continuous, the precise distribution is not important.

The space of genotypes together with their fitnesses form a labelled graph. If we further assume that the population initally consists of one single genotype, and that separate mutations never give rise to the same genotype, then the space of genotypes becomes a rooted tree. A selectively accessible or increasing path is then a simple path from the root to a leaf along vertices with increasing labels. For the House of Cards model in the SSWM regime, we may assume that the root has the genotype of minimal fitness. This leads us precisely to the accessibility percolation model outlined above.

\subsection{Other models}
Our methods could be extended to consider, for example, Galton-Watson trees instead of $n$-ary trees. 

Besides trees it is also natural to consider the House of Cards model on the $n$-dimensional hypercube $\{0,1\}^n$, for which there has been recent progress \cite{BBS, HM}. A selectively accessible path in this setting is a path of minimal length on increasing labels from $(0, \ldots, 0)$ to $(1, \ldots, 1)$. Both papers consider the effect of varying the fitness at the zero vertex on the number of accessible paths. Hegarty and Martinsson \cite{HM} obtain the threshold for the phase transition of the existence of increasing paths as $n \to \infty$. Berestycki, Brunet and Shi \cite{BBS} show that around this threshold, the number of such paths converges in distribution to the product of two independent exponential variables. As a first step, they obtain results for a particular rooted tree related to the hypercube.

Hegarty and Martinsson \cite{HM} also consider another model for the relationship between genotype and fitness, known as the \emph{Rough Mount Fuji} model in the evolutionary biology literature \cite{AUINKH}, where a linear drift, depending on the distance to the root, is introduced to the random fitnesses. This model on $n$-ary trees was also considered in \cite{NK}.

\subsection{Notation}
Throughout, we assume without loss of generality that the distribution of the labels is $U[0,1]$, and use the following crude double bound for Stirling's approximation valid for all $n \geq 1$:
\[
 2 < \frac{n!}{\sqrt{n} (n/e)^n} < 3.
\]
We also assume that $n = \alpha h$, rather than use unwieldy $\lfloor \cdot\rfloor$ notation all the way through the article. Since there is a clear monotonicity in $\alpha$ in the model, no extra difficulty arises in considering cases when $\alpha h$ is not an integer.

Let $P$ be the set of simple (that is, non-backtracking) paths from the root to a leaf in the tree; then $\#P = n^h$. For a path $u \in P$, write $X(u) = (X(u_1), \ldots, X(u_h))$ for the (i.i.d., $U[0,1]$) labels on its vertices. For any two paths $u,v \in P$, let $a(u,v) = \max\{k : u_k = v_k\}$. Clearly $X(u_j) = X(v_j)$ for all $j \leq a(u,v)$.

Define
\[
 I = \left\{ (x_1, \ldots, x_h) \in [0,1]^h : x_1 < x_2 < \ldots <x_h \right\},
\]
and for $\e \in [0,1)$,
\[
 C_\e = \left\{ (x_1, \ldots, x_h) \in [0,1]^h : x_j\geq \e \ \forall j \right\}
\]
and
\[
 D_\e = \left\{ (x_1, \ldots, x_h) \in [0,1]^h : x_j \geq \e + (1 - \e) \left( \frac{j-1}{h} \right) \ \forall j \right\}.
\]
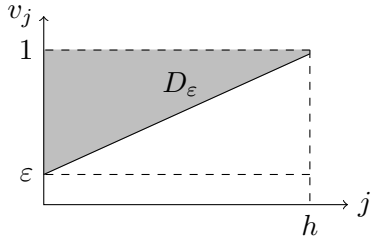
\[
\begin{tikzpicture}
 \draw [lightgray, fill] (0,2.05) -- (3.5,2.05) -- (3.5,2) -- (0,0.4) -- (0,2.05);
 \node at (1.8,1.6) {$D_\e$};
 \draw [<->] (0,2.5) -- (0,0) -- (4,0) node [right] {$j$};
 \node [left] at (0,2.5) {$v_j$};
 \draw [dashed] (3.5,2.05) -- (3.5,0) node [below] {$h$};
 \draw [dashed] (3.5,2.05) -- (0,2.05) node [left] {$1$};
 \draw [dashed] (3.5,0.4) -- (0,0.4) node [left] {$\e$};
 \draw [-] (3.5,2) -- (0,0.4);
\end{tikzpicture}
\]
Define
\[N_\e = \sum_{u\in P}\1_{\{X(u)\in I\cap D_\e\}},\]
and
\[N = \sum_{u\in P}\1_{\{X(u)\in I\}}.\]

\subsection{Outline of proof}\label{S:outline}
We will concentrate for the most part on proving Theorem 1, and then show how to adapt our proof to obtain Theorem 2.

We first observe that for a path $u$, $\{ X(u) \in I \}$ is the event that $h$ i.i.d.\ labels are in increasing order, which has probability $\frac1{h!}$. As $\#P = n^h$,
\[\E[N] = \frac{n^h}{h!}.\]
Using Stirling's approximation we have that
\begin{equation}\label{E:stirl}
\E[N] \asymp \frac{n^he^h}{\sqrt{h}h^h} = \frac{(\alpha e)^h}{\sqrt{h}}.
\end{equation}
In particular, we see that for $\alpha \leq 1/e$, $\E[N] \to 0$ as $h \to \infty$, and recover the $\alpha\leq 1/e$ part of Theorem \ref{T:mainthm} via Markov's inequality.

Nowak and Krug \cite{NK} gave this argument, and then went on to give an upper bound on $\E[N^2]$, which they used to get a lower bound on the probability that $N\geq 1$. We take a similar but slightly more subtle route, in that we will work for the most part with $N_\e$, whose moments are slightly harder to estimate but give us more information. Of course we have
\[\E[N_\e] \leq \E[N] = \frac{n^h}{h!}\asymp \frac{(\alpha e)^h}{\sqrt{h}}.\]
In Section \ref{S:1stmom} we will show that
\[\E[N_\e] \geq \frac{(\alpha(1-\e)e)^h}{3h^{3/2}}\]
and in Section \ref{S:2ndmom} we will see that when $\alpha(1-\e)e>1$ and $h$ is large,
\[\E[N_\e^2] \leq \E[N_\e] + \E[N_\e]^2 + c(\alpha(1-\e)e)^{2h}.\]
This will be enough to tell us that the probability that there is at least one path in $N_\e$ is at least a constant times $h^{-3}$ when $h$ is large.

We then do a fairly standard trick to complete the proof of Theorem \ref{T:mainthm} in Section \ref{S:mainthm}. We will show that there are many more than $h^3$ ``good'' subpaths in the first few levels of the tree: these are subpaths whose labels are increasing and small on the first few levels. Each of these subpaths then has a constant times $h^{-3}$ probability of being the start of an increasing path to a leaf.

Finally in Section \ref{S:ext} we show how our techniques can be fine-tuned to give Theorem \ref{T:ext}.

\section{First moment bound}\label{S:1stmom}
We aim to prove our lower bound on the first moment of $N_\e$:
\begin{prop}\label{P:1stmom}
\[\E[N_\e] \geq \frac{(\alpha(1-\e)e)^h}{3h^{3/2}}.\]
\end{prop}

We shall need the following lemma.

\begin{lem}\label{L:Integrals}
Let $U_1, \ldots, U_j$ be i.i.d.\ $U[0,1]$ random variables. Then
\[
 \P \left( U_1 \leq \ldots \leq U_j, \ U_1 \geq \frac{1}{j + 1}, \ldots, U_j \geq \frac{j}{j + 1} \right) = \frac{1}{(j+1)!}.
\]
\end{lem}

\begin{proof}
Let
\[
p = \P \left( U_1 \leq \ldots \leq U_j, \ U_1 \geq \frac{1}{j + 1}, \ldots, U_j \geq \frac{j}{j + 1} \right)
\]
and for each $i=2,\ldots,j$, define
\[
 I_i = \int_{\frac{j}{j + 1}}^1 \int_{\frac{j - 1}{j + 1}}^{v_j} \ldots \int_{\frac{i}{j + 1}}^{v_{i + 1}} \left( \frac{v_i^{i - 1}}{(i - 1)!} -  \frac{v_i^{i - 2}}{(j + 1)(i - 2)!} \right) \ dv_i \ldots dv_j.
\]
Note that
\[
 p = \int_{\frac{j}{j + 1}}^1 \int_{\frac{j - 1}{j + 1}}^{v_j} \ldots \int_{\frac{1}{j + 1}}^{v_2} 1 \ dv_1 \ldots dv_j = \int_{\frac{j}{j + 1}}^1 \int_{\frac{j - 1}{j + 1}}^{v_j} \ldots \int_{\frac{2}{j + 1}}^{v_3} \left( v_2 - \frac{1}{j + 1} \right) \ dv_2 \ldots dv_j = I_2.
\]
But for each $i=2,\ldots,j-1$,
\begin{align*}
 I_i & = \int_{\frac{j}{j + 1}}^1 \int_{\frac{j - 1}{j + 1}}^{v_j} \ldots \int_{\frac{i}{j + 1}}^{v_{i + 1}} \left( \frac{v_i^{i - 1}}{(i - 1)!} -  \frac{v_i^{i - 2}}{(j + 1)(i - 2)!} \right) \ dv_i \ldots dv_j \\
 & = \int_{\frac{j}{j + 1}}^1 \int_{\frac{j - 1}{j + 1}}^{v_j} \ldots \int_{\frac{i + 1}{j + 1}}^{v_{i + 2}} \left[ \frac{v_i^i}{i!} -  \frac{v_i^{i - 1}}{(j + 1)(i - 1)!} \right]_{\frac{i}{j + 1}}^{v_{i + 1}} \ dv_{i + 1} \ldots dv_j \\
 & = \int_{\frac{j}{j + 1}}^1 \int_{\frac{j - 1}{j + 1}}^{v_j} \ldots \int_{\frac{i + 1}{j + 1}}^{v_{i + 2}} \left( \frac{v_{i + 1}^i}{i!} -  \frac{v_{i + 1}^{i - 1}}{(j + 1)(i - 1)!} \right) \ dv_{i + 1} \ldots dv_j = I_{i + 1}.
\end{align*}
Therefore
\begin{multline*}
 p = I_2 = I_j = \int_{\frac{j}{j + 1}}^1 \left( \frac{v_j^{j - 1}}{(j - 1)!} -  \frac{v_j^{j - 2}}{(j + 1)(j - 2)!} \right) \ dv_j = \left[ \frac{v_j^j}{j!} -  \frac{v_j^{j - 1}}{(j + 1)(j - 1)!} \right]_{\frac{j}{j + 1}}^1 \\
 = \frac{1}{j!} -  \frac{1}{(j + 1)(j - 1)!} = \frac{j + 1 - j}{(j + 1)!} = \frac{1}{(j + 1)!}
\end{multline*}
as claimed.
\end{proof}

\begin{proof}[Proof of Proposition \ref{P:1stmom}]
By the fact that a $U[0,1]$ random variable conditioned to be at least $\e$ is a $U[\e,1]$ random variable,
\[
\E[N_\e] = n^h\P(U\in I\cap D_\e) = n^h \P(U\in I\cap D_\e | U\in C_\e)\P(U\in C_\e) = (\alpha h (1-\e))^h \P(U\in I\cap D_0).
\]
But by Lemma \ref{L:Integrals},
\[
\P(U\in I\cap D_0) \geq \P(U_1\leq 1/h)\P\left(U_2<\ldots<U_h, \hs U_i \geq \frac{i-1}{h} \hs \forall i=2,\ldots,h\right) = \frac{1}{h\cdot h!}.
\]
Applying Stirling's approximation once more, we obtain
\[\E[N_\e] \geq \frac{(\alpha(1-\e)e)^h}{3h^{3/2}}.\qedhere\]
\end{proof}

\section{Second moment bound}\label{S:2ndmom}
We now aim to prove an upper bound on the second moment of $N_\e$:
\begin{prop}\label{P:2ndmom}
If $\alpha(1-\e)e>1$, then there exists some constant $c > 0$ such that
\[\E[N_\e^2] \leq \E[N_\e] + \E[N_\e]^2 + c(\alpha(1-\e)e)^{2h}.\]
\end{prop}

\begin{proof}
We break the second moment into a sum over $k$-\emph{forks}:
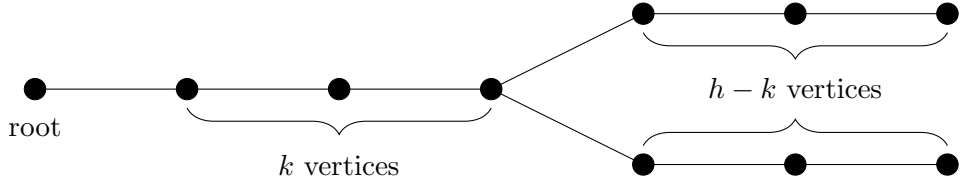
\[
\begin{tikzpicture}
 \draw [-] (0,0) -- (6,0);
 \draw [-] (12,1) -- (8,1) -- (6,0) -- (8,-1) -- (12,-1);
 \draw [fill] (0,0) circle (4pt) node at (0,-0.5) {root};
 \draw [fill] (2,0) circle (4pt);
 \draw [fill] (4,0) circle (4pt);
 \draw [fill] (6,0) circle (4pt);
 \draw [decorate,decoration={brace,amplitude=10}] (6,-0.25) -- (2,-0.25) node at (4,-1) {$k$ vertices};
 \draw [fill] (8,-1) circle (4pt);
 \draw [fill] (10,-1) circle (4pt);
 \draw [fill] (12,-1) circle (4pt);
 \draw [fill] (8,1) circle (4pt);
 \draw [fill] (10,1) circle (4pt);
 \draw [fill] (12,1) circle (4pt);
 \draw [decorate,decoration={brace,amplitude=10}] (12,0.75) -- (8,0.75) node at (10,0) {$h - k$ vertices};
 \draw [decorate,decoration={brace,amplitude=10}] (8,-0.75) -- (12,-0.75);
\end{tikzpicture}
\]
To this end, for $k=0,\ldots,h$, let
\[N_\e^2(k) = \sum_{\substack{u,v\in P:\\ a(u,v)=k}}\1_{\{X(u),X(v)\in I\cap D_\e\}}.\]
Then
\[N_\e^2 = \sum_{k=0}^h N_\e^2(k).\]
Clearly $N_\e^2(h) = N_\e$, and $\E[N_\e^2(0)] = \E[N_\e]^2$.

Let $U = (U_1,\ldots,U_h)$ and $V=(V_1,\ldots,V_h)$ each be a sequence of i.i.d.\ $U[0,1]$ random variables such that $U_j = V_j$ for all $j\leq k$ and $U_j$ and $V_j$ are independent for $j>k$. Using the fact that a uniform $[0,1]$ random variable conditioned to have value at least $\e$ is a uniform $[\e,1]$ random variable, we have for $k=2,\ldots,h-1$,
\begin{align*}
\E[N_\e^2(k)] &= n^k\cdot n(n-1)\cdot n^{2h-2k-2}\cdot\P(U,V\in I\cap D_\e)\\
&= \left(\frac{n-1}{n}\right) n^{2h-k}\P(U,V\in I\cap D_\e | U,V\in C_\e)\P(U,V\in C_\e)\\
&= \left(\frac{n-1}{n}\right) (\alpha h)^{2h-k}(1-\e)^{2h-k} \P(U,V\in I\cap D_0).
\end{align*}
Now,
\begin{align*}
\P(U,V\in I\cap D_0) &= \int_{\frac{k - 1}{h}}^1 \P(U,V \in I\cap D_0 | U_k = x) \ dx\\
&\leq \int_{\frac{k - 1}{h}}^1 \P(U_1<U_2<\ldots<U_{k-1}<x)\P(x<U_{k+1}<U_{k+2}<\ldots<U_h)^2 \ dx\\
&= \int_{\frac{k - 1}{h}}^1 \frac{x^{k-1}}{(k-1)!}\cdot \frac{(1-x)^{2h-2k}}{(h-k)!^2} \ dx.
\end{align*}
The curve $x^{k-1}(1-x)^{2h-2k}$ is decreasing on $x>(k-1)/(2h-k+1)$, so since $(k-1)/(2h-k+1) < (k-1)/h$,
\[
\int_{\frac{k - 1}{h}}^1 \frac{x^{k-1}}{(k-1)!}\cdot \frac{(1-x)^{2h-2k}}{(h-k)!^2} \ dx \leq \frac{((k-1)/h)^{k-1}}{(k-1)!}\cdot \frac{((h-k+1)/h)^{2h-2k}}{(h-k)!^2}.
\]
Putting these estimates together and then applying Stirling's approximation, we obtain that for $k=2,\ldots,h-1$,
\begin{align*}
\E[N_\e^2(k)] &\leq (\alpha h(1-\e))^{2h-k} \cdot \frac{((k-1)/h)^{k-1}}{(k-1)!}\cdot\frac{((h-k+1)/h)^{2h-2k}}{(h-k)!^2}\\
&\leq (\alpha(1-\e))^{2h-k} h \cdot \frac{e^{k-1}}{2(k-1)^{1/2}} \cdot \frac{e^{2h-2k+2}}{4(h-k+1)}\\
& = \frac{e}{8} \cdot \frac{(\alpha(1-\e)e)^{2h-k}h}{(k-1)^{1/2}(h-k+1)}.
\end{align*}

Similarly,
\begin{align*}
\E[N_\e^2(1)] &\leq n^{2h-1}\frac{(1-\e)^{2h-1}}{(h-1)!^2}\\
&\leq \frac{e}{4} (\alpha(1-\e)e)^{2h-1}.
\end{align*}

Thus if $\alpha (1 - \e)e > 1$, for some constant $c$,
\begin{align*}
\E[N_\e^2] &\leq \E[N_\e] + \E[N_\e]^2 + \frac{e}{4} (\alpha(1-\e)e)^{2h-1} + \sum_{k=2}^{h-1} \frac{e}{8} \cdot \frac{(\alpha(1-\e)e)^{2h-k}h}{(k-1)^{1/2}(h-k+1)}\\
&\leq \E[N_\e] + \E[N_\e]^2 + c(\alpha(1-\e)e)^{2h}.\qedhere
\end{align*}
\end{proof}

\section{Proof of Theorem \ref{T:mainthm}}\label{S:mainthm}
As noted previously, it suffices to prove a lower bound when $\alpha>1/e$. Choose $\e\in(0,1)$ such that $\alpha(1-\e)e>1$. By the Paley-Zygmund inequality,
\[\P\left(N_\e \geq \frac{\E[N_\e]}{2}\right) \geq \frac{\E[N_\e]^2}{4\E[N_\e^2]}.\]
By Proposition \ref{P:1stmom}, if we choose $\delta\in(0,\log(\alpha(1-\e)e))$ then $\E[N_\e]/2 \geq e^{\delta h}$ for all large $h$, so
\[\P\left(N_\e > \exp(\delta h)\right) \geq \frac{\E[N_\e]^2}{4\E[N_\e^2]}.\]
But by Propositions \ref{P:1stmom} and \ref{P:2ndmom}, for large $h$ and some constant $c'$,
\[\E[N_\e^2] \leq c'h^3 \E[N_\e]^2.\]
Thus we get
\begin{equation}\label{E:estim}
\P(N_\e > \exp(\delta h)) \geq \frac{1}{4c'h^3}.
\end{equation}
Of course, we now want to improve this bound to get something exponentially close to $1$ on the right-hand side. To do this, we will consider the first four levels of the tree separately from the rest. The idea is that with high probability, there are $\sim n^4$ paths from the root of length $4$ whose labels are increasing and $< \e$. Each vertex at level $4$ then has a subtree of $(h-4)^n$ paths of length $h-4$ and with probability $\gtrsim h^{-3}$ lots of these subpaths have labels which are increasing and $> \e$, by (\ref{E:estim}). So the probability that no path is increasing should look like, up to constants, $(1-h^{-3})^{n^4}$, which decays exponentially as desired. We note that our choice of four levels is only to counteract the factor of $h^{-3}$ in (\ref{E:estim}) and working with any finite number of levels greater than $3$ would also suffice.

We start by considering the subpaths $v$ from the root to level $4$. Although one can count subpaths whose labels satisfy $X(v_1) < \ldots < X(v_4) < \e$, we will instead count subpaths whose labels lie inside a priori intervals, allowing us to consider levels one at a time. More precisely, for $j \leq 4$, let $M_j$ be the set of subpaths $v$ from the root to level $j$ such that $X(v_i) \in [(i-1)\e/4,i\e/4)$ for each $i=1,\ldots,j$. Observe that $\# M_1$ is the sum of $n$ independent Bernoulli random variables of parameter $\e / 4$; similarly, for $2 \leq j \leq 4$, given $\# M_{j-1}\geq k$, $\# M_j$ is at least a sum of $kn$ independent Bernoulli random variables of parameter $\e / 4$. For this reason, the following well-known form of the Chernoff bound will be useful.

\begin{lem}\label{L:Chernoff}
Let $Z_1, \ldots, Z_r$ be independent Bernoulli random variables and let $Z = \sum_{i=1}^r Z_i$. Then
\[
 \P \left( Z \leq \frac{\E[Z]}{2} \right)\leq \exp \left( -\frac{\E[Z]}{8} \right).
\]
\end{lem}

We can now prove our desired bound on $\# M_4$.

\begin{lem}\label{L:Level4}
\[
 \P(\# M_4 \leq (n \e/8)^4) \leq 4\exp(-n\e^4 / 16384).
\]
\end{lem}

\begin{proof}
At level 1, there are $n$ vertices, and $\E[\# M_1] = n \e /4$. Thus by Lemma \ref{L:Chernoff},
\[
 \P\left(\# M_1 \leq n \e/8 \right) \leq \exp \left( - \frac{n \e}{4\cdot8} \right).
\]
At level 2, given that $\# M_1 > n \e/8$, there are at least $n \e^2/8$ vertices whose parent had label in $[0,\e/4)$, and so
\[
 \E[\# M_2 \ | \ \# M_1 > n \e/8] \geq \frac{n^2 \e^2}{8 \cdot 4}.
\]
Again by Lemma \ref{L:Chernoff},
\[
 \P\left(\left.\# M_2 \leq (n \e/8)^2 \hsl\right|\hsl \# M_1 > n \e/8\right) \leq \exp \left( -\frac{(n \e/8)^2}{4} \right).
\]
Similarly,
\[
 \P\left(\left.\# M_3 \leq (n \e/8)^3 \hsl\right|\hsl \# M_2 > (n \e/8)^2 \right) \leq \exp \left( -\frac{(n \e/8)^3}{4} \right).
\]
and
\[
 \P\left(\left.\# M_4 \leq (n \e/8)^4 \hsl\right|\hsl \# M_3 > (n \e/8)^3 \right) \leq \exp \left( -\frac{(n \e/8)^4}{4} \right).
\]
Summing these estimates gives the result.
\end{proof}

To complete the proof of Theorem \ref{T:mainthm}, note that
\[\P(N \leq \exp(\delta h)) \leq \P(\# M_4 \leq (n \e/8)^4 ) + \P(N \leq \exp(\delta h), \hsl \# M_4 > (n \e/8)^4).\]
Suppose that $u\in M_4$, and consider the subtree of height $h-4$ rooted at the vertex $u_4$. In order that $N\leq e^{\delta h}$, it must hold that there are no more than $e^{\delta h}$ paths in this subtree that have labels ordered and greater than $\e$. But we know from (\ref{E:estim}), since $n/(h-4) \geq n/h = \alpha$, that the probability of this event is at most $1-c'h^{-3}$. Thus, applying also Lemma \ref{L:Level4} and the inequality $1+x \leq e^{x}$,
\[\P(N \leq \exp(\delta h)) \leq 4\exp(-n \e^4 / 16384) + (1-c'h^{-3})^{(n \e/8)^4} \leq  \exp(-\eta h)\]
for some $\eta > 0$, which proves Theorem \ref{T:mainthm}.

\section{Extension to $\alpha = 1/e + o(1)$: proof of Theorem \ref{T:ext}}\label{S:ext}
We now turn our attention to the case when $n = \alpha_h h$ where $\alpha_h = (1+\beta_h)/e$, $\beta_h \to 0$. For the first part of the theorem, it is not difficult to see from (\ref{E:stirl}) and Markov's inequality that if $\log h - 2h\beta_h \to\infty$, then the probability that there exists an increasing path tends to $0$ as $h\to\infty$. For the second part, choose $\e_h$ such that $\e_h/\beta_h\to 0$ but $h\e_h/\log h\to\infty$ as $h\to\infty$. Then
\[
 \alpha_h (1-\e_h) e = (1+\beta_h)(1-\e_h) = 1+\beta_h - \e_h - \beta_h \e_h.
\]
So, for $h$ sufficiently large, we have $\alpha_h (1-\e_h) e > 1$ and the proofs of Propositions \ref{P:1stmom} and \ref{P:2ndmom} go through almost unchanged. As before we get
\[\E[N_{\e_h}] \geq \frac{(\alpha_h(1-\e_h)e)^h}{3h^{3/2}}\to\infty\]
and
\[
 \E[N_{\e_h}^2] \leq \E[N_{\e_h}] + \E[N_{\e_h}]^2 + \frac{e}{4} (\alpha_h(1-\e_h)e)^{2h-1} + \sum_{k=2}^{h-1} \frac{e}{8} \cdot \frac{(\alpha_h(1-\e_h)e)^{2h-k}h}{(k-1)^{1/2}(h-k+1)}.
\]
However, since $(\alpha_h(1-\e_h)e)$ is not constant, we cannot bound the last term, up to constants, by $(\alpha_h(1-\e_h)e)^{2h}$ as before. Instead, we see that
\[
 \sum_{k=2}^{h-1} \frac{1}{(k-1)^{1/2}(h-k+1)} < \int_0^{h-1} \frac{1}{\sqrt x (h-x)} \, dx = \frac{2}{\sqrt h} \tanh^{-1} \left( \sqrt{\frac{h - 1}{h}} \right) \sim \frac{\log h}{\sqrt h},
\]
which leads to a bound, up to constants, of $h^{1/2}\log h (\alpha_h(1 - \e_h)e)^{2h}$. Therefore,
\[
 \E[N_{\e_h}^2] \leq \E[N_{\e_h}] + \E[N_{\e_h}]^2 + ch^{1/2}\log h(\alpha_h(1 - \e_h)e)^{2h} \leq c' h^{15/4}\E[N_{\e_h}]^2
\]
for some constants $c$ and $c'$.

Now, the main difficulty arises in our application of Lemma \ref{L:Level4}. With the new exponent of $15/4$, applying our previous argument, we get
\[
 \P(N \leq \exp(\delta h), \hsl \# M_4 > (n \e_h/8)^4) \leq (1 -c'h^{-15/4})^{(n \e_h/8)^4} \leq \exp(-c''h^{1/4}\e_h^4)
\]
for some constant $c''$. However, unlike before, this probability does not converge to $0$ as $h \to \infty$, because $h^{1/4}\e_h^4$ does not necessarily converge to infinity. So instead of working with a fixed number of levels, we work with $\log h$ levels. More precisely, we work with $\lfloor \log h\rfloor$ levels, but to save notation we simply write $\log h$.

Now, for $j=1,\ldots,\log h$, we define $\tld M_j$ to be the set of subpaths $v$ from the root to level $j$ such that $X(v_i) \in [(i-1)\e_h/\log h,i\e_h/\log h)$ for each $i=1,\ldots,j$. As before, by repeatedly applying Lemma \ref{L:Chernoff} we get
\[
 \P\left( \left. \# \tld M_j \leq \left(\frac{n \e_h}{2\log h}\right)^j \hsl\right|\hsl \# \tld M_{j - 1} > \left(\frac{n \e_h}{2\log h}\right)^{j - 1} \right) \leq \exp \left(- \frac14 \left(\frac{n \e_h}{2 \log h}\right)^j\right),
\]
for $j = 1, \ldots, \log h$. Summing these bounds gives
\[
 \P\left( \# \tld M_{\log h} \leq \left(\frac{n \e_h}{2\log h}\right)^{\log h} \right) \leq \sum_{j = 1}^{\log h} \exp \left(- \frac14 \left(\frac{n \e_h}{2 \log h}\right)^j  \right)
\]
which converges to $0$ as $h \to \infty$ by our assumption that $h\e_h/\log h \to \infty$. Therefore, with high probability, we have at least $(n \e_h / 2\log h)^{\log h} = h^{\log(n \e_h / 2\log h)}$ ``good'' increasing subpaths up to level $\log h$, each of which has probability at least $c'h^{-15/4}$ of extending to an increasing path to a leaf. Since
\[
 \displaystyle (1 - c'h^{-15/4})^{(n \e_h / 2\log h)^{\log h}} \leq \exp \left(- c'h^{-15/4 + \log(n \e_h / 2\log h)}\right) \to 0,
\]
the probability that there is no increasing path from the root to a leaf converges to $0$ as $h \to \infty$, completing the proof of Theorem \ref{T:ext}.

\paragraph{Acknowledgements.}

The first author is partly supported by EPSRC grant EP/K007440/1. This collaboration began during the 2013 UK probability meeting in Cambridge supported by EPSRC grant EP/I03372X/1 and continued at the Mathematisches Forschungsinstitut Oberwolfach supported by the ESF and Oberwolfach Leibniz Graduate Student programme. Both authors thank two anonymous referees for their helpful comments.


\end{document}